\documentclass[12pt]{amsart}

\usepackage{amsmath,amssymb,amscd,amsthm,graphicx,enumerate}

\usepackage{hyperref}
\hypersetup{breaklinks=true}

\numberwithin{equation}{section}
\theoremstyle{plain}

\makeatletter
\newtheorem*{rep@theorem}{\rep@title}
\newcommand{\newreptheorem}[2]{%
\newenvironment{rep#1}[1]{%
 \def\rep@title{#2 \ref{##1}}%
 \begin{rep@theorem}}%
 {\end{rep@theorem}}}
\makeatother

\newtheorem{theorem}[equation]{Theorem}
\newreptheorem{theorem}{Theorem}

\newtheorem{proposition}[equation]{Proposition}
\newtheorem{lemma}[equation]{Lemma}
\newtheorem{corollary}[equation]{Corollary}

\newtheorem{claim}[equation]{Claim}

\theoremstyle{remark}
\newtheorem{remark}[equation]{Remark}

\theoremstyle{definition}

\newtheorem*{question*}{Question}

\newtheorem{example}[equation]{Example}
\newcommand{\RR}{\mathbb{R}}

\newcommand{\abs}[1]{\lvert#1\rvert}

\newcommand{\R}{\mathbb R}

\newcommand{\dist}{\operatorname{dist}}

\newcommand{\eps}{\varepsilon}

\providecommand{\abs}[1]{\lvert #1\rvert}

\def\XXint#1#2#3{{\setbox0=\hbox{$#1{#2#3}{\int}$}
     \vcenter{\hbox{$#2#3$}}\kern-.5\wd0}}

\begin{document}

\title[Curve shortening flow]{Lectures on curve shortening flow}
\author{Robert Haslhofer}\thanks{These notes are from the last three weeks of my PDE II courses at the University of Toronto in Spring 2016, 2021 and 2026. I thank all the students for their questions and feedback. My research has been supported by an NSERC discovery grant.}

\begin{abstract}
The curve shortening flow is a geometric heat equation for curves and provides an accessible setting to illustrate many important concepts from nonlinear partial differential equations, including maximum principle estimates, monotonicity formulas, Harnack inequalities and blowup analysis. All these techniques will be combined to give an exposition of Huisken's proof of Grayson's beautiful theorem that the curve shortening flow shrinks any closed embedded curve in the plane to a round point.
\end{abstract}

\maketitle

\tableofcontents

\section{Curve shortening flow basics}

A one-parameter family of embedded curves $\{\Gamma_t\subset \R^{2}\}_{t\in I}$ moves by curve shortening flow if the normal velocity at each point is given by the curvature vector, namely
\begin{equation}\label{eq_mcf}
\partial_t p  = \vec{\kappa}(p)
\end{equation}
for all $p\in \Gamma_t$ and all $t\in I$. Here, $I$ is an interval, $\partial_t p$ is the normal velocity at $p$, and $\vec{\kappa}(p)$ is the curvature vector at $p$.

\begin{example}[Round shrinking circles] If $\Gamma_t=\partial B^{2}_{r(t)}\subset\R^{2}$, then \eqref{eq_mcf} reduces to an ODE for the radius, namely
$\dot{r}=-1/r$.
The solution with $r(0)=R$ is given by $r(t)=\sqrt{R^2-2t}$, $t\in (-\infty,R^2/2)$.
\end{example}

\begin{example}[Grim reaper]
Another explicit solution is given by $\Gamma_t=\textrm{graph}(-\log \cos p)+t$ where $p\in(-\tfrac{\pi}{2},\tfrac{\pi}{2})$ and $t\in(-\infty,\infty)$.
\end{example}

\begin{example}[Paperclip and hairclip]
Despite the nonlinear nature of the curve shortening flow, the upwards translating grim reaper given by $e^{-y(t)}=e^{-t}\cos x(t)$ and the downwards translating grim reaper given by $e^{y(t)}=e^{-t}\cos x(t)$ can be combined to give another pair of solutions given implicitly as the solution set of
\begin{equation}\label{eq_paperclip}
\cosh y(t)=e^{-t}\cos x(t),
\end{equation}
respectively
\begin{equation}\label{eq_hairclip}
\sinh y(t)=e^{-t}\cos x(t).
\end{equation}
The paperclip, given as solution of \eqref{eq_paperclip} restricted to $\abs{x}< \pi/2$, describes a compact ancient solution that for $t\to 0$ becomes extinct in a round point and for $t\to -\infty$ looks like two copies of the grim reaper glued together. The hairclip \eqref{eq_hairclip} is an eternal solution, which for $t\to -\infty$ looks like an infinite row of alternatingly translating up and down grim reapers, and for $t\to +\infty$ converges to a horizontal line.
\end{example}

From now on we will focus on the evolution of closed embedded curves. It is often most convenient to describe the evolution \eqref{eq_mcf} in terms of a one-parameter family of embeddings
\begin{equation}
\gamma=\gamma(\cdot,t): S^1 \times I \rightarrow \RR^{2}
\end{equation}
with $\Gamma_t=\gamma(S^1,t)$. Setting $p=\gamma(x,t)$, the curve shortening flow equation then takes the form
\begin{equation}\label{eq_mcf2}
\partial_t \gamma(x,t) =  \kappa(x,t)N(x,t),
\end{equation}
where we have expressed the curvature vector $\vec{\kappa}$ as a product of the curvature $\kappa$ and the inward pointing unit normal vector $N$.

\begin{remark} The evolution can also be written in the form
\begin{equation}
\partial_t \gamma = \partial^2_s\gamma,
\end{equation}
where $s$ denotes arclength. Even though this almost looks like the linear heat equation, the curve shortening flow is of course a nonlinear PDE, since the arclength $s$ depends in a nonlinear way on $x$ and $t$.
\end{remark}

We will for now assume that we have a smooth solution of \eqref{eq_mcf2} on a time interval $[0,T)$ and derive various elementary consequences.\footnote{The assumption that the initial curve is smooth is not really necessary. As a consequence of the estimates we will discuss any $C^2$-curve becomes instantaneously smooth under the curve shortening flow. In fact, the Cauchy problem for the curve shortening flow is well posed starting with any finite length Jordan curve \cite{Lauer}.}

Let $A(t)$ be the area of the domain enclosed by $\Gamma_t$. Then
\begin{equation}
\frac{d}{dt}A(t)=-\int_{\Gamma_t} \kappa \, ds = -2\pi, 
\end{equation}
and thus $A(t)=A(0)-2\pi t$. In particular, $T\leq A(0)/2\pi$. We will see in the final section that the evolution actually can be continued smoothly until time $A(0)/2\pi$, where the flow becomes extinct in a round point (this is a big theorem originally due to Grayson \cite{Grayson}, combined with earlier work in the convex case by Gage-Hamilton \cite{Gage_Hamilton}).

A geometric incarnation of the maximum principle is the following comparison principle: If $\{\Gamma_t\}_{t\in [t_0,t_1]}$ and $\{\Gamma_t'\}_{t\in [t_0,t_1]}$ are two curve shortening flows (say at least one of them compact) that are initially disjoint, then they stay disjoint under the evolution, i.e.
\begin{equation}
\Gamma_{t_0}\cap\Gamma'_{t_0}=\emptyset \qquad \Rightarrow \qquad \Gamma_{t}\cap\Gamma'_{t}=\emptyset \quad \forall t\geq t_0.
\end{equation}
In fact, it is easy to check that $t\mapsto\dist(\Gamma_t,\Gamma'_t)$ is nondecreasing in time. In particular, any closed embedded curve contained in $B_1$, even a curve spiralling around millions of times, will unwind itself and become round in time $T\leq \frac{1}{2}$. This vividly illustrates the strength of Grayson's theorem and the power of geometric heat equations in general.

We will next compute how the total length $L(t)=\int_{\Gamma_t}ds$ evolves under curve shortening flow. To this end, we start by computing
\begin{align}
\partial_t \langle\partial_x\gamma, \partial_x\gamma\rangle^{1/2} = \langle\partial_x\partial_t\gamma,T\rangle =\langle\partial_x(\kappa N),T\rangle,
\end{align}
where $T=\partial_x \gamma/ \abs{\partial_x \gamma}$ denotes the unit tangent and where we used the evolution equation \eqref{eq_mcf2} in the last step. Since $\langle N,T\rangle =0$ and since $\langle\partial_x N,T\rangle= -\kappa \frac{ds}{dx}$ by the definition of curvature, we obtain that
\begin{equation}\label{eq_gradientflow}
\frac{d}{dt}L(t)=\frac{d}{dt}\int_{S^1} \langle\partial_x\gamma, \partial_x\gamma\rangle^{1/2}\, dx =-\int_{\Gamma_t} \kappa^2 ds.
\end{equation}
Equation \eqref{eq_gradientflow} has the interpretation that the curve shortening flow is the gradient flow of the length functional, i.e. indeed deserves its name since it shortens curves as efficiently as possible (in general the variation of the length functional with normal speed $v$ is given by $-\int_\Gamma \kappa v ds$).

\begin{proposition}[Evolution of curvature]\label{prop_evol_eq}
If $\{\Gamma_t\subset \R^{2}\}$ evolves by curve shortening flow, then its curvature evolves by
\begin{equation}\label{eq_evol_curv}
\kappa_t=\kappa_{ss}+\kappa^3.
\end{equation}
\end{proposition}

\begin{proof}
For convenience we work with a parametrization that satisfies $\abs{\partial_x \gamma}=1$ and $\langle\partial_x^2 \gamma, T\rangle=0$ at the point $(x,t)$ under consideration. Since $\kappa=\abs{\partial_x \gamma}^{-2}\langle \partial^2_x \gamma, N\rangle$, at the point $(x,t)$ we can compute
\begin{align}
\kappa_t &= \partial_t\langle \partial^2_x \gamma,N\rangle-2\langle T,\partial_x\partial_t\gamma\rangle \langle \partial^2_x \gamma,N\rangle\nonumber\\
 &= \langle \partial_x^2\partial_t \gamma,N\rangle-2\kappa\langle T,\partial_x\partial_t\gamma\rangle,
\end{align}
where we used that $\partial_tN$ is tangential and $\partial_x^2 \gamma$ is normal. Remembering our evolution equation \eqref{eq_mcf2} and using $\langle \partial_x N,N\rangle = 0$ we infer that
\begin{align}
\partial_t \kappa &= \partial_x^2\kappa + \kappa \langle \partial_x^2 N ,N\rangle-2\kappa^2\langle T,\partial_x N\rangle\nonumber\\
 &= \partial_x^2\kappa - \kappa \langle \partial_x N , \partial_x N\rangle+2\kappa^3.
\end{align}
Noting that $\partial_x^2 \kappa = \kappa_{ss}$ and $\langle \partial_x N , \partial_x N\rangle=\kappa^2$ at the point $(x,t)$, this proves the proposition.
\end{proof}

Using Proposition \ref{prop_evol_eq} and the maximum principle we obtain:

\begin{corollary}[Convexity]
Convexity is preserved under curve shortening flow, i.e. if $\kappa>0$ at $t=0$ then $\kappa>0$ for all $t\in[0,T)$.
\end{corollary}

More precisely, if $\kappa_{\min}(t):=\min_{\Gamma_t}\kappa$ is positive at $t=0$, then it is nondecreasing in time and satisfies
\begin{equation}
\kappa_{\min}(t)\geq \frac{\kappa_{\min}(0)}{1-2t\kappa_{\min}^2(0)}.
\end{equation}
In particular, this gives the (non sharp) estimate $T\leq 1/(2\kappa^2_{\min}(0))$.

We finish this first lecture by proving that the supremum of the curvature controls all the higher derivatives:

\begin{theorem}[Derivative estimates]\label{thm_der_est}
There exist constants $C_\ell=C_\ell(K,T)<\infty$ such that if $\{\Gamma_t\}_{t\in [0,T)}$ is a solution of the curve shortening flow with $\sup_{t\in [0,T)}\sup_{\Gamma_t}\abs{\kappa}\leq K$, then
\begin{equation}
\sup_{\Gamma_t} \abs{\partial^\ell_s \kappa}\leq \frac{C_\ell}{t^{\ell/2}}.
\end{equation}
\end{theorem}

\begin{proof}[Proof (sketch)] Using Proposition \ref{prop_evol_eq} we compute
\begin{equation}\label{eq_kappasq}
(\partial_t-\partial^2_s)\kappa^2=-2\kappa_s^2+2\kappa(\partial_t-\partial^2_s)\kappa = -2\kappa_s^2+2\kappa^4.
\end{equation}
Next, differentiating \eqref{eq_evol_curv} with respect to arc length we obtain
\begin{align}
(\kappa_t)_s = \kappa_{sss}+3\kappa^2\kappa_s,
\end{align}
which together with the commutator identity $(\kappa_s)_t=(\kappa_t)_s+\kappa^2\kappa_s$ implies
\begin{equation}
(\partial_t-\partial_s^2)\kappa_s=4\kappa^2\kappa_s,
\end{equation}
and thus
\begin{equation}\label{kappa_ssq}
(\partial_t-\partial_s^2)\kappa_s^2=-2\kappa_{ss}^2+8\kappa^2\kappa_s^2.
\end{equation}
Combining the evolution equations \eqref{eq_kappasq} and \eqref{kappa_ssq} we obtain
\begin{equation}
(\partial_t-\partial^2_s)(t\kappa_s^2+\beta\kappa^2)\leq (8t \kappa^2+1-2\beta)\kappa_s^2+2\beta\kappa^4\leq 2\beta K^4,
\end{equation}
provided we choose $\beta\geq (8TK^2+1)/2$. Thus, the maximum principle implies that
\begin{equation}
t\kappa_s^2 \leq \beta K^2 +2\beta K^4 T,
\end{equation}
which proves the derivative estimate for $\ell=1$.

For general $\ell$, by induction one obtains the differential inequality
\begin{equation}
(\partial_t-\partial_s^2)\abs{\partial_s^\ell\kappa}^2\leq-2\abs{\partial_s^{\ell+1}\kappa}^2+\alpha_\ell \left(\sum_{i+j+k=\ell}\abs{\partial_s^i \kappa}\abs{\partial_s^j \kappa}\abs{\partial_s^k \kappa}\right) \abs{\partial_s^\ell \kappa},
\end{equation}
where $\alpha_\ell$ are some numerical constants. The derivative estimates then follow by considering the evolution of
\begin{equation}
F_\ell= t^\ell \abs{\partial_s^\ell \kappa}^2+\sum_{i=0}^{\ell-1} \beta_{\ell,i}t^i \abs{\partial_s^i \kappa}^2
\end{equation}
for suitable constants $\beta_{\ell,i}$ and arguing by induction. The details are left as an exercise.
\end{proof}

\section{Existence and uniqueness}

The goal of this lecture is to explain how to prove existence and uniqueness for the curve shortening flow.

\begin{theorem}[Existence and Uniqueness]\label{thm_ex_uniq}
Let $\gamma_0:S^1\to\mathbb{R}^2$ be an embedded curve. Then there exists a unique smooth solution $\gamma: S^1\times [0,T)\to \mathbb{R}^2$ of the curve shortening flow
\begin{equation}
\partial_t\gamma=\partial_s^2\gamma,\qquad\quad \gamma|_{t=0}=\gamma_0,
\end{equation}
defined on a maximal time interval [0,T). The maximal existence time is characterized by curvature blowup, namely
\begin{equation}
\sup_{S^1\times [0,T)}\abs{\kappa}(x,t)=\infty.
\end{equation}
\end{theorem}

We start by explaining that the curve shortening flow is not strictly parabolic. If $\gamma:S^1\times [0,T)\to \mathbb{R}^2$ evolves by curve shortening flow then
\begin{align}
\partial_t\gamma&=\partial^2_s\gamma = \abs{\partial_x\gamma}^{-1}\partial_x(\abs{\partial_x\gamma}^{-1}\partial_x\gamma)\\
&= \abs{\partial_x\gamma}^{-2}\left(\partial_x^2\gamma-\left\langle \frac{\partial_x \gamma}{\abs{\partial_x\gamma}},\partial_x^2\gamma\right\rangle\frac{\partial_x \gamma}{\abs{\partial_x\gamma}}\right).
\end{align}
In components $\gamma=(\gamma^1,\gamma^2)$ this reads
\begin{equation}
 \left(\begin{array}{c}
\partial_t\gamma^1 \\
\partial_t\gamma^2
\end{array}\right)=
\abs{\partial_x\gamma}^{-4}\left(\begin{array}{cc}
(\partial_x\gamma^2)^2 & \partial_x\gamma^1 \partial_x \gamma^2\\
\partial_x\gamma^1 \partial_x \gamma^2 & (\partial_x\gamma^1)^2
\end{array}\right)
\left(\begin{array}{c}
\partial_x^2\gamma^1 \\
\partial_x^2\gamma^2
\end{array}\right).
\end{equation}
Note that the matrix is positive semidefinite, but not positive definite (it has positive trace, but vanishing determinant). Thus, the standard theory for strictly parabolic systems cannot be applied.

Degeneracies like above are typical for geometric PDEs. For the curve shortening flow, the underlying reason is that geometrically only the normal component of the velocity is meaningful, i.e. the velocity is only determined up to tangential motion / change of parametrization. To overcome this degeneracy we have to fix a gauge. To this end, we represent the flow as evolving graph over $\gamma_0$ (if $\gamma_0$ is not smooth, one can instead choose a smooth curve nearby), i.e. we make the ansatz
\begin{equation}
\tilde{\gamma}(x,t)=\gamma_0(x)+u(x,t)N(x).
\end{equation} 
Assuming that $\gamma_0$ is parametrized by arc length, we have
\begin{equation}
\tilde{\gamma}'=u'N+(1-ku)T,
\end{equation}
where $k$ is the curvature of $\gamma_0$, and $T$ and $N$ are the unit tangent and unit normal of $\gamma_0$, respectively. Differentiating again we obtain
\begin{equation}
\tilde{\gamma}''=(u''+k(1-ku))N-(k'u+2ku')T.
\end{equation}
Using the formula $\kappa=\abs{\tilde{\gamma}'}^{-3}\det(\tilde{\gamma}',\tilde{\gamma}'')$ we obtain the curvature of $\tilde{\gamma}$,
\begin{equation}
\kappa=\frac{(1-ku)u''+2ku'^2+k'uu'-2k^2u+k^3u^2+k}{((1-ku)^2+u'^2)^{3/2}}.
\end{equation}
Next, the unit normal vector $\tilde{N}$ of $\tilde{\gamma}$ is obtained by rotating $\tilde{\gamma}'/\abs{\tilde{\gamma}'}$ by $\pi/2$ and thus equal to
\begin{equation}
\tilde{N}=\frac{(1-ku)N-u'T}{((1-ku)^2+u'^2)^{1/2}}.
\end{equation}
Finally, observing that $\Gamma_t=\tilde{\gamma}(S^1,t)$ moves by curve shortening if and only if $\langle \tilde{N}, u' N\rangle =\kappa$, we obtain the evolution equation
\begin{equation}\label{eq_quasilin}
u_t = \frac{u''+(1-ku)^{-1}(2ku'^2+k'uu'-2k^2u+k^3u^2+k)}{((1-ku)^2+u'^2)}.
\end{equation}
Equation \eqref{eq_quasilin} is a quasilinear strictly parabolic equation (as long as say $\abs{ku}\leq 1/2$), and thus has a unique solution on some time interval $[0,\eps)$, c.f. the previous PDE lectures (this can be done e.g by using Picard iteration combined with the theory for the inhomogenous linear heat equation and energy estimates in $H^j$ for $j$ sufficiently large).

In general, $\tilde{\gamma}$ only solves the equation up to tangential motion, i.e.
\begin{equation}
\partial_t\tilde{\gamma}=\kappa\tilde{N}+f\partial_x\tilde{\gamma},
\end{equation}
for some function $f=f(x,t)$. To get a parametrization $\gamma(x,t)$ that literally solves the equation $\partial_t\gamma=\partial_s^2\gamma$, we have to set
\begin{equation}
\gamma(x,t)=\tilde{\gamma}(\varphi_t(x),t),
\end{equation}
where $\varphi_t:S^1\to S^1$ is the unique solution of the family of ODEs
\begin{equation}
\tfrac{d}{dt} \varphi_t(x)=-f(x,t) \frac{\partial_x\tilde{\gamma}(x,t)}{\partial_x\tilde{\gamma}(\varphi_t(x),t)}, \qquad \varphi_0(x)=x.
\end{equation}

This yields short time existence and uniqueness for the curve shortening flow (recall also that embeddedness is preserved). Finally, suppose that $\{\Gamma_t\}_{t\in [0,T)}$ is a solution on a maximal time interval $[0,T)$, but
\begin{equation}
\limsup_{t\to T}\sup_{\Gamma_t}\abs{\kappa}<\infty.
\end{equation}
Then, by Theorem \ref{thm_der_est} all derivatives of $\kappa$ are bounded as well. We can thus pass to a smooth limit $\Gamma_T$, and applying the short time existence result we can continue the evolution until time $T+\eps$; this contradicts the fact that $T$ was maximal and finishes the proof of Theorem \ref{thm_ex_uniq}.

\section{Huisken's monotonicity formula and applications}

Recall that under curve shortening flow the length evolves by
\begin{equation}\label{ev_of_length}
\frac{d}{dt}\int_{\Gamma_t} ds = -\int_{\Gamma_t} \kappa^2\, ds.
\end{equation}
However, since $\textrm{Length}(\lambda \Gamma)=\lambda \textrm{Length}(\Gamma)$, this is not that useful when considering blowup sequences with $\lambda\to \infty$. A great advance was made by Huisken, who discovered a scale invariant monotone quantity.
To describe this, let $\{\Gamma_t\subset \R^{2}\}$ be a curve shortening flow (say of closed curves, or complete curves with at most polynomial length growth),
let $X_0=(x_0,t_0)$ be a point in space-time, and let
\begin{equation}
 \rho_{X_0}(x,t)=(4\pi(t_0-t))^{-1/2} e^{-\frac{\abs{x-x_0}^2}{4(t_0-t)}}\qquad (t<t_0),
\end{equation}
be the $1$-dimensional backwards heat kernel centered at $X_0$.

\begin{theorem}[Huisken's monotonicity formula \cite{Huisken_monotonicity}]\label{thm_huisken_mon}
\begin{equation}\label{eq_huisken_mon}
 \frac{d}{dt}\int_{\Gamma_t} \rho_{X_0}\, ds = -\int_{\Gamma_t} \left|\kappa+\frac{\langle\gamma,N\rangle}{2(t_0-t)}\right|^2 \rho_{X_0}\, ds \qquad (t<t_0).
\end{equation}
\end{theorem} 

Huisken's monotonicity formula \eqref{eq_huisken_mon} can be thought of as weighted version of \eqref{ev_of_length}. A key property is its invariance under parabolic rescaling (cf. homework). Moreover, the equality case of \eqref{eq_huisken_mon} exactly characterizes the selfsimilarly shrinking solutions (cf. homework).

\begin{proof}
Without loss of generality let $X_0=(0,0)$. The proof essentially amounts to deriving belows pointwise identity \eqref{eq_pointwise} for $\rho=\rho_0$.

Since the tangential gradient of $\rho$ is given by $\partial_s\rho = \nabla\rho-\langle \nabla\rho, N\rangle N$,
the intrinsic Laplacian of $\rho$ can be expressed as
\begin{equation}
\partial_s^2\rho =\langle T,\nabla_T\partial_s\rho\rangle
=\langle T,\nabla_T\nabla\rho\rangle+\kappa\langle N, \nabla\rho\rangle.
\end{equation}
Observing also that $\tfrac{d}{dt}\rho=\partial_t \rho+\kappa \langle N, \nabla\rho\rangle$ along the flow, we compute
\begin{align}
(\tfrac{d}{dt}+\partial_s^2)\rho&=\partial_t\rho+\langle T,\nabla_T\nabla\rho\rangle+2\kappa \langle N, \nabla\rho\rangle\nonumber\\
&=\partial_t \rho+\langle T,\nabla_T\nabla\rho\rangle+\frac{\langle N,\nabla \rho\rangle^2}{\rho}-\left|{\kappa-\frac{\langle N ,\nabla\rho\rangle}{\rho}}\right|^2 \rho+\kappa^2\rho.
\end{align}
We can now easily check that $\partial_t \rho+\langle T,\nabla_T\nabla\rho\rangle+\frac{\langle N, \nabla \rho\rangle^2}{\rho}=0$. Thus 
\begin{equation}\label{eq_pointwise}
\left(\frac{d}{dt}+\partial^2_s-\kappa^2\right)\rho=-\left|\kappa-\frac{\langle\gamma,N\rangle}{2t}\right| ^2 \rho.
\end{equation}
Using also the evolution equation $\tfrac{d}{dt}ds =-\kappa^2ds $, we conclude that
\begin{equation}
 \frac{d}{dt}\int_{\Gamma_t} \rho\, ds = -\int_{\Gamma_t} \left|\kappa-\frac{\langle\gamma,N\rangle}{2t}\right|^2 \rho \, ds   \qquad (t<0).
\end{equation}
This proves the theorem.
\end{proof}

We will now explain how Huisken's monotonicity formula can be used to study singularities via blowup analysis. Let $\{\Gamma_t\subset\mathbb{R}^2\}_{t\in [0,T)}$ be a curve shortening flow of closed embedded curves, defined on a maximal time interval $[0,T)$. From the previous lecture we know that the curvature blows up at the singular time, i.e.
\begin{equation}\label{eq_curv_blow_up}
\limsup_{t\to T}\max_{\Gamma_t}\abs{\kappa}=\infty.
\end{equation}
In the following we will assume that the singularity forms with the so called type I blowup rate
\begin{equation}\label{typeone}
\max_{\Gamma_t}\abs{\kappa}\leq \frac{C}{\sqrt{T-t}},
\end{equation}
for some $C<\infty$ (this assumption will be justified in later lectures).

We say that $x_0\in\mathbb{R}^2$ is a blowup point if there are sequences $t_i\to T$, $p_i\in\Gamma_{t_i}$ such that $\abs{\kappa}(p_i)\to\infty$ and $p_i\to x_0$. By \eqref{eq_curv_blow_up} there indeed exists at least one blowup point $x_0$. We now rescale parabolically with center $(x_0,T)$, i.e. for $\lambda>0$ consider the rescaled flow
\begin{equation}
\Gamma^{\lambda}_t:=\lambda\cdot\left(\Gamma_{T+\lambda^{-2}t}-x_0\right),\qquad t\in [-\lambda^2T,0).
\end{equation}

\begin{claim}
For $\lambda\to\infty$ the flows $\{\Gamma_t^\lambda\}_{t\in [-\lambda^2T,0)}$ converge smoothly to the family of round shrinking circles $\{\partial B_{\sqrt{-2t}}(0)\}_{t\in(-\infty,0)}$.
\end{claim}

\begin{proof}
For the rescaled flow the blowup rate \eqref{typeone} takes the form
\begin{equation}
\max_{\Gamma_t^\lambda}\abs{\kappa}\leq \frac{C}{\sqrt{-t}},\qquad t\in [-\lambda^2T,0).
\end{equation}
In particular, given any compact time interval $I\subset (-\infty,0)$, the flow $\{\Gamma_t^\lambda\}$ is defined on $I$ for $\lambda$ large enough, and has uniformly bounded curvature on $I$. By the derivative estimates (Theorem \ref{thm_der_est}) we also have locally uniform bounds for all the derivatives of the curvatures.
Thus, for any sequence $\lambda_i\to \infty$ we can find a subsequence $\lambda_{i_k}$ such that $\{\Gamma_t^{\lambda_{i_k}}\}$ converges smoothly to a limit $\{\Gamma_t\}_{t\in (0,\infty)}$.

We will now analyze the limit $\{\Gamma_t\}_{t\in (0,\infty)}$: By construction, the limit is an ancient solution of the curve shortening flow. The limit is embedded with multiplicity one (this follows e.g. from the quantitative embeddedness estimate from Lecture 5). Using the definition of blowup point and comparison with round shrinking circles we see that $\Gamma_{-1}^\lambda\cap B_2(x_0)\neq \emptyset$ for $\lambda$ large enough. Thus, the limit is nonempty. Now, by Huisken's monotonicity formula for every $t_1<t_2<0$ we have
\begin{equation}
\int_{t_1}^{t_2}\int_{\Gamma_t^\lambda} \left|\kappa-\frac{\langle\gamma,N\rangle}{2t}\right|^2 \rho\, ds\, dt =
\left[-\int_{\Gamma_t} \rho_{(x_0,T)}\, ds\right]_{T-t_1/\lambda^2}^{T-t_2/\lambda^2}\to 0
\end{equation}
as $\lambda\to\infty$. Thus, $\{\Gamma_t\}_{t\in (0,\infty)}$ is selfsimilarly shrinking and completely determined by its time slice $\Gamma_{-1}$ which satisfies
\begin{equation}
\kappa+\frac{\langle\gamma,N\rangle}{2}=0.
\end{equation}
Since $x_0$ is a blowup point, $\Gamma_{-1}$ cannot be a straight line by the local regularity theorem (see Theorem \ref{app_thm_easy_brakke} below). Thus, $\Gamma_{-1}$ must be a round circle of radius $\sqrt{2}$ (c.f. homework).

Finally, by uniqueness of the blowup limit, $x_0$ is unique and the subsequential convergence actually entails full convergence.
\end{proof}

To finish this lecture, let us discuss the local regularity theorem which says if the density
\begin{equation}
\Theta(\{\Gamma_t\},(x_0,t_0),r):=\int_{\Gamma_{t_0}-r^2 } \rho_{(x_0,t_0)}\,  ds
\end{equation}
is close to one, then the curvature is controlled. As usual, we write $X=(x,t)$ for points in space-time, and denote by $P_r(X)=B_r(x)\times (t-r^2,t]$ the parabolic ball with center $X$ and radius $r$.

\begin{theorem}[Local regularity theorem \cite{brakke,white_regularity}]\label{app_thm_easy_brakke}
There exist universal constants $\eps>0$ and $C<\infty$ with the following property:
If $\{\Gamma_t\subset \mathbb{R}^2\}_{t\in (t_0-2r^2,t_0]}$ is a curve shortening flow (say of closed curves, or complete curves with at most polynomial length growth) with
\begin{equation}
 \sup_{\overline{X}_0\in P_r(X_0)}\Theta (\{\Gamma_t\},\overline{X}_0,r)<1+\eps,
\end{equation}
then
\begin{equation}
 \sup_{P_{r/2}(X_0)}\abs{\kappa}\leq {C}r^{-1}.
\end{equation}
\end{theorem}

\begin{proof}
Suppose the assertion fails. Then there exist a sequence of curve shortening flows $\{\Gamma^j_t\subset \mathbb{R}^2\}_{t\in (-2,0]}$, with
\begin{equation}\label{blowup_easy_brakke}
 \sup_{\overline{X}_0\in P_1(0)}\Theta (\{\Gamma_t\},\overline{X}_0,1)<1+j^{-1},
\end{equation}
but such that there are points $X_j\in P_{1/2}(0)$ with $\abs{\kappa}(X_j)> j$.

We now select points $Y_j\in P_{3/4}(0)$ with $K_j=\abs{\kappa}(Y_j)> j$, such that
\begin{equation}\label{app_brakke_point_sel}
 \sup_{P_{{j}/{(10K_j)}}(Y_j)}\abs{\kappa}\leq 2 K_j.
\end{equation}
Let us explain how this point selection works: Fix $j$. If $Y^0_j=X_j$ already satisfies (\ref{app_brakke_point_sel}) with $K^0_j=\abs{\kappa}(Y^0_j)$, we are done. Otherwise, there is a point $Y^1_j\in P_{j/(10K^0_j)}(Y_j^0)$ with $K^1_j=\abs{\kappa}(Y^1_j)>2K^0_j$.
If $Y^1_j$ satisfies (\ref{app_brakke_point_sel}), we are done. Otherwise, there is a point $Y^2_j\in P_{j/(10K^1_j)}(Y_j^1)$ with $K^2_j=\abs{\kappa}(Y^2_j)>2K^1_j$, etc.
Note that $\frac{1}{2}+\frac{j}{10K_j^0}(1+\frac{1}{2}+\frac{1}{4}+\ldots)<\frac{3}{4}$. By smoothness, the iteration terminates after a finite number of steps, and the last point of the iteration lies in $P_{3/4}(0)$ and satisfies (\ref{app_brakke_point_sel}).

Continuing the proof of the theorem, let $\{\hat{\Gamma}_t^j\}$ be the flows obtained by shifting $Y_j$ to the origin and parabolically rescaling by $K_j=\abs{\kappa}(Y_j)$.
Since the rescaled flow satisfies $\abs{\kappa}(0)=1$ and $\sup_{P_{j/10}(0)}\abs{\kappa}\leq 2$, we can pass smoothly to a global limit. On the one hand, the limit is non-flat. On the other hand, by the rigidity case of Huisken's monotonicity formula and equation \eqref{blowup_easy_brakke} the the limit is a straight line; a contradiction.
\end{proof}

\section{Hamilton's Harnack inequality}

\begin{theorem}[Hamilton's Harnack inequality \cite{Hamilton_harnack}]\label{thm_hamilton_harnack}
If $\{\Gamma_t\subset\mathbb{R}^2\}_{t\in [0,T)}$ is a convex solution of the curve shortening flow (say closed or complete with bounded curvature) then
\begin{equation}
\frac{\kappa_t}{\kappa}-\frac{\kappa_s^2}{\kappa^2}+\frac{1}{2t}\geq 0.
\end{equation}
\end{theorem}

\begin{proof}
The proof is very similar to the one of the Li-Yau Harnack inequality \cite{LiYau}. It suffices to consider the case where the solution existed since time $-\eps$ and $\kappa\geq \eps$. Let $f:=\log \kappa$ and $F:=t (f_s^2-f_t)$. We want to use the maximum principle to show that $F\leq 1/2$ for all $t\in [0,T)$. Note that $F\leq 1/2$ for small $t$. We compute
\begin{align}
F_{ss}&=t(2f_sf_{sss}+2f_{ss}^2- (f_{t})_{ss})\nonumber\\
&=t(2f_sf_{sss}+2f_{ss}^2- (f_{ss})_t+2\kappa^2f_{ss}+2\kappa^2f_s^2),
\end{align}
where we used the commutator formula $[\partial_t,\partial_s]=\kappa^2\partial_s$.
Using the evolution equation for $\kappa$ we see that $f_{ss}=-F/t-\kappa^2$ and thus
\begin{multline}
F_{ss}=-2f_sF_{s}+\tfrac{2F^2}{t}-\tfrac{F}{t}+F_t\\
-4t\kappa^2f_s^2+4F\kappa^2+2t\kappa^4+2t\kappa^2 f_t-2\kappa^2F-2t\kappa^4+2t\kappa^2f_s^2.
\end{multline}
Miraculously, the nonlinear terms cancel and the quantity on the last line is identically zero, i.e.
\begin{equation}
F_{ss}-F_t=-2f_sF_{s}+\tfrac{1}{t}F(2F-1)
\end{equation}
If there is a maximum point $(x_0,t_0)$ with $F(x_0,t_0)>1/2$, then
\begin{equation}
0\geq (F_{ss}-F_t)|_{(x_0,t_0)}\geq 0 + \tfrac{1}{t}F(x_0,t_0)(2F(x_0,t_0)-1)>0;
\end{equation}
a contradiction. This proves the theorem.
\end{proof}

Applying the Harnack inequality with $t$ replaced by $t-t_0$ and taking the limit $t_0\to -\infty$ we obtain:

\begin{corollary}
If $\{\Gamma_t\subset\mathbb{R}^2\}_{t\in (-\infty,T)}$ is an ancient convex solution of the curve shortening flow then
\begin{equation}
\frac{\kappa_t}{\kappa}-\frac{\kappa_s^2}{\kappa^2}\geq 0,
\end{equation}
in particular $\kappa_t \geq 0$.
\end{corollary}

\begin{theorem}[Translating solitons \cite{Hamilton_harnack}]\label{thm_hamilton_harnack}
Any eternal strictly convex solution $\{\Gamma_t\subset\mathbb{R}^2\}_{t\in (-\infty,\infty)}$  of the curve shortening flow such that $\kappa$ has a critical point somewhere in space time, must be a translating soliton, i.e. there exists some vector $V\in \mathbb{R}^2$ such that $\Gamma_t = \Gamma_0+t V$.
\end{theorem}

\begin{remark}
The only strictly convex translating soliton, up to scaling and rigid motion, is the grim reaper (c.f. homework).
\end{remark}

\begin{proof}
Assume $\kappa$ has a critical point at $(x_0,t_0)$, i.e. $\kappa_t=0=\kappa_s$ at $(x_0,t_0)$. The Harnack quantity
\begin{equation}
Z=\frac{\kappa_t}{\kappa}-\frac{\kappa_s^2}{\kappa^2}
\end{equation}
satisfies $Z\geq 0$ and $Z(x_0,t_0)=0$. Using the strict maximum principle (c.f. the evolution equations in the above proof), we see that $Z\equiv 0$ for all $t\leq t_0$, i.e.
\begin{equation}\label{time_der_kappa}
\kappa_t = \frac{\kappa_s^2}{\kappa}.
\end{equation}
Consider the vector field
\begin{equation}
V:=-\frac{\kappa_s}{\kappa}T+\kappa N.
\end{equation}
We compute
\begin{equation}
V_s=\left(-\frac{\kappa_{ss}}{\kappa}+\frac{\kappa_s^2}{\kappa^2}-\kappa^2\right)T+(\kappa_s - \kappa_s)N.
\end{equation}
Using equation \eqref{time_der_kappa} we see that $V_s\equiv 0$. Similarly, $V_t\equiv 0$. Thus $V$ is a constant vector. Since the normal component of $V$ is given by $\kappa N$, this implies that $\Gamma_t = \Gamma_{t_0}+(t-t_0) V$ for $t\leq t_0$. By uniqueness of the curve shortening flow we conclude that  $\Gamma_t = \Gamma_0+t V$ for all $t$.
\end{proof}

\section{Huisken's distance comparison principle}

In this lecture, we discuss Huisken's estimate for the ratio between the intrinsic and extrinsic distance along the curve shortening flow. This can be thought of as quantitative version of the fact that embeddedness is preserved along the curve shortening flow.
Let $X:S^1\times [0,T)\to \mathbb{R}^2$ be a family of embedded curves evolving by curve shortening flow. Let $L(t)$ be the total length of the curve at time $t$. Given two points $x,y\in S^1$, denote by $\ell(x,y,t)$ the intrinsic distance between $X(x,t)$ and $X(y,t)$, and by $d(x,y,t)=\abs{X(x,t)-X(y,t)}$ the extrinsic distance. Following Huisken, we consider the quantity
\begin{equation}
R(t):=\sup_{x\neq y} \frac{L(t)}{\pi d(x,y,t)}\sin\frac{\pi \ell(x,y,t)}{L(t)}.
\end{equation}

\begin{remark}
Note that $R\geq 1$, and $R=1$ only on the round circle.
\end{remark}

\begin{remark} Note that for $\ell\ll L$ we get $\frac{L}{\pi d}\sin\frac{\pi \ell}{L}\sim \frac{\ell}{d}$, i.e. the ratio between intrinsic and extrinsic distance. Since $\sin(\frac{\pi}{2}+\varphi)=\sin(\frac{\pi}{2}-\varphi)$, the function $\sin\frac{\pi \ell(x,y,t)}{L(t)}$ is smooth even at points with $\ell(x,y,t)=L(t)/2$.

\end{remark}

\begin{theorem}[Huisken's distance comparison principle \cite{Huisken_distance}]\label{thm_huisken_dist}
If a family of closed embedded curves in the plane evolves by curve shortening flow, then the function $R(t)$ is nonincreasing in time.
\end{theorem}

\begin{proof}
Let us describe the proof following \cite{Huisken_distance} and \cite{Brendle_twopoint}. If the assertion is false, we can find times $t_1<t_2$ and a real number $r>1$ such that
\begin{equation}\label{ass_real_r}
R(t_1)<r \qquad\textrm{and}\qquad R(t_2)>r.
\end{equation}
We now consider the function
\begin{equation}
Z(x,y,t):= r d(x,y,t)-\frac{L(t)}{\pi}\sin \frac{\pi \ell(x,y,t)}{L(t)}.
\end{equation}
By \eqref{ass_real_r} there exists a time $\bar{t}\in (t_0,t_1)$ and a pair of points $\bar{x}\neq \bar{y}$ such that $Z(\bar{x},\bar{y},\bar{t})=0$, and $Z(x,y,t)\geq 0$ for all $x,y\in S^1$ and all $t\in(t_0,\bar{t})$.

Without loss of generality we can assume that the parametrization at time $\bar{t}$ is by arclength, and that the orientation is chosen such that $\partial_x\ell (\bar{x},\bar{y},\bar{t})=-1$ and $\partial_y\ell (\bar{x},\bar{y},\bar{t})=+1$. 

We start by computing the first derivatives
\begin{equation}
\frac{\partial Z}{\partial x}(x,y,\bar{t}) = r \, \frac{\langle X(x,\bar{t})-X(y,t),\frac{\partial X}{\partial x}(x,\bar{t}) \rangle}{|X(x,\bar{t})-X(y,\bar{t})|} + \cos \frac{\pi \, \ell(x,y,\bar{t})}{L(\bar{t})},
\end{equation}
and 
\begin{equation}
\frac{\partial Z}{\partial y}(x,y,\bar{t}) = -r \, \frac{\langle X(x,\bar{t})-X(y,t),\frac{\partial X}{\partial y}(y,\bar{t}) \rangle}{|X(x,\bar{t})-X(y,\bar{t})|} - \cos \frac{\pi \, \ell(x,y,\bar{t})}{L(\bar{t})}.
\end{equation}
These first derivatives vanish when evaluated at $(\bar{x},\bar{y},t)$. In particular, adding these two identities gives
\begin{equation}\label{sum_first_partials}
\Big \langle X(\bar{x},\bar{t})-X(\bar{y},\bar{t}),\frac{\partial X}{\partial x}(\bar{x},\bar{t})-\frac{\partial X}{\partial y}(\bar{y},\bar{t}) \Big \rangle = 0.
\end{equation}

To keep the notation concise, we write $T(\bar x)=\frac{\partial X}{\partial x}(\bar{x},\bar{t})$,  $T(\bar y)=\frac{\partial X}{\partial y}(\bar{y},\bar{t})$,  $\kappa(\bar x)N(\bar x)=\frac{\partial^2 X}{\partial x^2}(\bar{x},\bar{t})$ and $\kappa(\bar y)N(\bar y)=\frac{\partial^2 X}{\partial y^2}(\bar{y},\bar{t})$, and also use the abbreviations $d=d(\bar{x},\bar{y},\bar{t})$, $\ell=\ell(\bar{x},\bar{y},\bar{t})$, $L=L(\bar{t})$, and
\begin{equation}
\omega=\frac{X(\bar{y},\bar{t})-X(\bar{x},\bar{t})}{d}.
\end{equation}
Using this notation, the identity \eqref{sum_first_partials} can be rewritten as
\begin{equation}
\langle \omega, T(\bar x)\rangle = \langle \omega, T(\bar y)\rangle.
\end{equation}
We next compute the second order partial $x$-derivatives of $Z$:
\begin{align*} 
\frac{\partial^2 Z}{\partial x^2}(\bar{x},\bar{y},\bar{t})= \frac{r}{d}\left(1-\langle \omega, T(\bar x)\rangle^2\right)-r\kappa(\bar x)\langle \omega, N(\bar x)\rangle+\frac{\pi}{L}\sin\frac{\pi \ell}{L}.
\end{align*}
Similarly,
\begin{align*} 
\frac{\partial^2 Z}{\partial y^2}(\bar{x},\bar{y},\bar{t})= \frac{r}{d}\left(1-\langle \omega, T(\bar y)\rangle^2\right)+r\kappa(\bar y) \langle \omega, N(\bar y)\rangle+\frac{\pi}{L}\sin\frac{\pi \ell}{L},
\end{align*}
and
\begin{align*} 
\frac{\partial^2 Z}{\partial x\partial y}(\bar{x},\bar{y},\bar{t})= -\frac{r}{d}\left( \langle T(\bar x),T(\bar y)\rangle-
\langle T(\bar x),\omega \rangle\langle \omega, T(\bar y)\rangle\right)-\frac{\pi}{L}\sin \frac{\pi\ell}{L}.
\end{align*}
Define $\alpha\in (0,\pi/2)$ by $\cos\alpha = \langle \omega, T(\bar x)\rangle=\langle \omega, T(\bar y)\rangle$. Then $\langle T(\bar x),T(\bar y)\rangle = \cos(2\alpha)$, and thus
\begin{multline}
\frac{\partial^2 Z}{\partial x^2}(\bar{x},\bar{y},\bar{t})+\frac{\partial^2 Z}{\partial y^2}(\bar{x},\bar{y},\bar{t})-2\frac{\partial^2 Z}{\partial x\partial y}(\bar{x},\bar{y},\bar{t})=\\
-r\kappa(\bar x) \langle \omega, N(\bar x)\rangle+r\kappa(\bar y) \langle \omega, N(\bar y)\rangle+\frac{4\pi}{L}\sin\frac{\pi\ell}{L}.
\end{multline}
Finally, the time derivative of $Z$ can be computed as
\begin{align*} 
\frac{\partial Z}{\partial t}(\bar{x},\bar{y},\bar{t}) 
&=- r\langle \omega, \kappa(\bar x) N(\bar x)-\kappa(\bar y) N(\bar y)\rangle\\
&+ \Big (\frac{1}{\pi} \sin \frac{\pi \ell}{L} - \frac{\ell}{L} \cos \frac{\pi \ell}{L} \Big ) \int_{S^1} \kappa^2 + \cos \frac{\pi \ell}{L} \int_{\bar{x}}^{\bar{y}} \kappa^2. 
\end{align*} 
Since $r > 1$ and $Z(\bar{x},\bar{y},\bar{t}) = 0$, the curve $X(S^1,\bar{t})$ cannot have constant curvature, and thus
\begin{equation}
\int_{S^1} \kappa^2 > \frac{1}{L} \, \bigg ( \int_{S^1} \kappa \bigg )^2 = \frac{4\pi^2}{L}.
\end{equation} 
Similarly,
\begin{equation}
\int_{\bar{x}}^{\bar{y}} \kappa^2 \geq \frac{1}{\ell} \, \bigg ( \int_{\bar{x}}^{\bar{y}} \kappa \bigg )^2 = \frac{4\alpha^2}{\ell}.
\end{equation}
Putting everything together, we conclude that
\begin{align*} 
0 &\geq \frac{\partial Z}{\partial t}(\bar{x},\bar{y},\bar{t}) - \frac{\partial^2 Z}{\partial x^2}(\bar{x},\bar{y},\bar{t}) - \frac{\partial^2 Z}{\partial y^2}(\bar{x},\bar{y},\bar{t}) +2 \, \frac{\partial^2 Z}{\partial x \, \partial y}(\bar{x},\bar{y},\bar{t}) \\ 
&> \frac{4}{\ell} \, \Big ( \alpha^2 - \frac{\pi^2 \ell^2}{L^2} \Big )\cos\frac{\pi \ell}{L}. 
\end{align*} 
On the other hand, the inequality $r > 1$ implies $\cos \alpha \leq \cos \frac{\pi \ell}{L}$, hence $\alpha \geq \frac{\pi \ell}{L}$. This is a contradiction.
\end{proof}

\begin{remark}
By the monotonicity we have $R(t)\leq C$, where $C:=R(0)<\infty$ measures the quantitative embeddedness of the initial curve. Note also that $R$ is scaling invariant. In particular, this implies that the grim reaper and the paperclip cannot arise as a blowup limit of a curve shortening flow of closed embedded curves.
\end{remark}

\section{Grayson's convergence theorem}

In this final lecture we explain Huisken's proof of Grayson's theorem:

\begin{theorem}[{Grayson's theorem \cite{Grayson}}]\label{thm_grayson}
If $\Gamma\subset \mathbb{R}^2$ is a closed embedded curve, then the curve shortening flow $\{\Gamma_t\}_{t\in [0,T)}$ with $\Gamma_0=\Gamma$ exists until $T=\frac{A_\Gamma}{2\pi}$ and converges for $t\to T$ to a round point, i.e. there exists a unique point $x_0\in \mathbb{R}^2$ such that the rescaled flows
\begin{equation}
\Gamma_t^\lambda:= \lambda\cdot (\Gamma_{T+\lambda^{-2}t}-x_0)
\end{equation}
converge for $\lambda\to\infty$ to the round shrinking circle $\{\partial B_{\sqrt{-2t}}\}_{t\in (-\infty,0)}$.
\end{theorem}

\begin{remark} Many different proofs of Grayson's theorem have been found by now. In particular, there is a nice geometric proof by Andrews-Bryan \cite{AndrewsBryan}, and a clever blowup argument by White \cite{Chodosh_lectures}.
\end{remark}

\begin{lemma}[Grayson]\label{lemma_grayson}
Along the curve shortening flow we have
\begin{equation}
\frac{d}{dt}\int_{\Gamma_t}\abs{\kappa}\, ds= -2\sum_{x:\kappa(x,t)=0}\abs{\kappa_s}(x,t)
\end{equation}
\end{lemma}

\begin{proof}[Proof of Lemma \ref{lemma_grayson}]
Since solutions of the curve shortening flow are analytic, there are only finitely many inflection points. We compute
\begin{equation}
\frac{d}{dt}\left(\int_{\{\kappa\geq 0\}} \kappa\, ds-\int_{\{\kappa\leq 0\}} \kappa\, ds\right)=\int_{\{\kappa\geq 0\}} \kappa_{ss}\, ds-\int_{\{\kappa\leq 0\}} \kappa_{ss}\, ds.
\end{equation}
Integrating by parts, the assertion follows.
\end{proof}

\begin{proof}[Proof of Theorem \ref{thm_grayson}]
Let $T<\infty$ be the maximal existence time of the curve shortening flow starting at $\Gamma$. Suppose towards a contradiction that
\begin{equation}\label{typetwo}
\limsup_{t\to T}\left((T-t)\max_{\Gamma_t}\kappa^2\right)=\infty.
\end{equation} 
We now perform a type II blowup as in \cite{huisken-sinestrari1}, see also \cite{Altschuler}. For any integer $k\geq 1/T$ let $t_k\in [0,T-1/k]$, $x_k\in S^1$ be such that
\begin{equation}
\kappa^2(x_k,t_k)(T-1/k-t_k)=\max_{t\leq T-1/k, \, x\in S^1} \kappa^2(x,t)(T-1/k-t).
\end{equation}
Furthermore we set
\begin{equation}
\lambda_k = \kappa(x_k,t_k),\quad t_k^{(0)}=-\lambda_k^2 t_k,\quad  t_k^{(1)}=\lambda_k^2(T-1/k-t_k).
\end{equation}
By \eqref{typetwo}, given any $M<\infty$ there exist $\bar{t}<T$ and $\bar{x}\in S^1$ such that $\kappa^2(\bar x,\bar t)(T-\bar t)>2M$. For $k$ large enough we have
\begin{equation}
\bar t < T - 1/k,\quad \kappa^2(\bar x,\bar t)(T-\bar t -1/k) > M.
\end{equation}
It follows that
\begin{equation}
t_k^{(1)}= \kappa^2(x_k,t_k)(T-1/k-t_k)\geq \kappa^2(\bar x,\bar t)(T-1/k-\bar t)> M.
\end{equation}
Since $t_k^{(1)}$ is increasing and $M$ is arbitrary, this implies $t_k^{(1)}\to\infty$. It follows that $\lambda_k\to\infty$, $t_k\to T$ and $t_k^{(0)}\to -\infty$.

We now consider the sequence of rescaled flows
\begin{equation}
\Gamma^k_t= \lambda_k\cdot \left(\Gamma_{t_k+\lambda_k^{-2}t}-x_k\right),\qquad t\in [t_k^{(0)},t_k^{(1)}).
\end{equation}
By construction, after choosing a suitable orientation, $\Gamma^k_t$ has $\kappa_k=1$ at $t=0$ at the origin. In addition, our choice of $(x_k,t_k)$ implies
\begin{equation}
\kappa^2_k(x,t)\leq \frac{T-1/k-t_k}{T-1/k-t_k-\lambda_k^2 t}=\frac{t_k^{(1)}}{t_k^{(1)}-t},\qquad t\in [t_k^{(0)},t_k^{(1)}).
\end{equation}
After passing to a subsequence, we thus get a smooth limit $\{\Gamma_t^\infty\}_{t\in (-\infty,\infty)}$. The limit satisfies $\kappa=1$ at time 0 at the origin, and $\kappa^2\leq 1$ at every point and every time. Moreover, by Lemma \ref{lemma_grayson} the limit satisfies
\begin{equation}
\int_{-\infty}^\infty \sum_{x:\kappa(x,t)=0}\abs{\kappa_s}(x,t)\, dt=0,
\end{equation}
i.e. if there was any point with $\kappa=0$ then we would have $\kappa_s=0$ at this point also. Using the evolution equations and analyticity this would imply that $\{\Gamma_t^\infty\}_{t\in (-\infty,\infty)}$ is a straight line \cite{Angenent}; a contradiction. Thus $\kappa>0$, and by the equality case of Hamilton's Harnack inequality from Lecture 4 and the classification of translating solitons from the homework $\{\Gamma_t^\infty\}_{t\in (-\infty,\infty)}$ must be a grim reaper; this contradicts the bound for the ratio between the intrinsic and extrinsic distance from Lecture 5.
We thus have proved the type I blowup rate bound
\begin{equation}
\limsup_{t\to T}\left((T-t)\max_{\Gamma_t}\kappa^2\right)<\infty.
\end{equation}
and by the results from Lecture 3 we conclude that $T=\frac{A_\Gamma}{2\pi}$ and that for $t\to T$ the flow converges to a round point.
\end{proof}

\vspace{10mm}

\bibliography{lectures_csf}

\bibliographystyle{alpha}

\vspace{10mm}

{\sc Department of Mathematics, University of Toronto,  40 St George Street, Toronto, ON M5S 2E4, Canada}\\

\end{document}